\documentclass[12pt]{article}
\usepackage[english]{babel}
\usepackage{amsfonts,amsmath,amsxtra,amsthm,amssymb,latexsym}

\theoremstyle{plain}
\newtheorem{theorem}{Theorem}[section]

\newtheorem{lemma}[theorem]{Lemma}

\newtheorem{definition}[theorem]{Definition}
\newtheorem{remark}[theorem]{Remark}
\newtheorem{example}[theorem]{Example}
\newtheorem{proposition}[theorem]{Proposition}

\numberwithin{equation}{section}

\DeclareMathOperator{\sgn}{sgn} 
\DeclareMathOperator{\loc}{loc}

\DeclareMathOperator{\disc}{disc}
\DeclareMathOperator{\ess}{ess}
\DeclareMathOperator{\ran}{ran}
\DeclareMathOperator{\dom}{dom}

\def\R{\mathbb R}
\def\C{\mathbb C}

\def\wt{\widetilde}

\def\vphi{\varphi}
\def\la{\lambda}
\def\H{\mathcal{H}}
\def\K{\mathcal{K}}
\def\pr{\mathcal{P}}
\def\uph{\upharpoonright}

\def\Hte{\mathcal H}
\def\btel{\langle}
\def\bter{\rangle}

\def\T{\mathcal A}

\begin{document}

\title{Abstract kinetic equations with positive collision operators}

\author{
I.M. Karabash \\
(the University of Calgary, Canada)}

\date{}

\maketitle

\begin{abstract}
We consider "forward-backward" parabolic equations in the abstract form
$Jd \psi / d x + L \psi = 0$, $ 0< x < \tau \leq \infty$, where $J$ and $L$ are operators in a Hilbert space $H$ such that $J=J^*=J^{-1}$, $L=L^* \geq 0$, and $\ker L = 0$.
The following theorem is proved: if the operator $B=JL$ is similar to a self-adjoint operator, then associated half-range boundary problems have unique solutions. 
We apply this theorem to corresponding nonhomogeneous equations, to the time-independent Fokker-Plank equation
$ \mu \frac {\partial \psi }{\partial x} (x,\mu ) =  
b(\mu) \frac {\partial^2 \psi}{\partial \mu^2} (x, \mu )$, \ $ 0<x<\tau$, \ $ \mu \in \R$, as well as to other parabolic equations of the "forward-backward" type.
The abstract 
kinetic equation $ T d \psi/dx = - A \psi (x) + f(x)$, where $T=T^*$
is injective and $A$ satisfies a certain positivity assumption, is considered also.
\end{abstract}

\quad

MSC-class: 47N55, 35K70, 47B50 (Primary) 35M10, 35K90 (Secondary)

\quad

Keywords: {forward-backward parabolic equations,  kinetic equations, transport equations, 
equations of mixed type, J-self-adjoint operator, similarity, regular and singular critical points}

\quad

\section{Introduction}

Consider the equation 
\begin{gather}
w(\mu) \frac {\partial \psi }{\partial x} (x,\mu ) =  
\frac {\partial^2 \psi}{\partial \mu^2} (x, \mu ) - 
q(\mu) \psi (x, \mu)  
\qquad ( 0<x<\tau<\infty, 
\ \mu \in \R ),
\label{e FBSL} 
\end{gather}
and the associated boundary value problem 
\begin{gather}
\psi (0, \mu) = \vphi_+ (\mu) \qquad \text{if} 
\quad \mu >0,
\qquad
\psi (\tau, \mu) = \vphi_- (\mu) \qquad \text{if} 
\quad \mu <0.
\label{e FBBC} 
\end{gather}
Here $w$ and $q$ are locally summable on $\R$ and $\mu w(\mu) > 0$,\ $ \mu \in \R$. So the weight function $w$ changes its sign at $0$.
We assume also that the Sturm-Liouville operator 
\begin{gather} \label{e L} 
L : y \mapsto \frac 1{|w|}( -y'' +qy) 
\end{gather}
defined on the maximal domain in the Hilbert space $L^2 (\R, |w(x)|dx)$ is self-adjoint.  
Boundary value problems of this (forward-backward) type arise as various 
kinetic equations (e.g., \cite{BP83,KLH82,GvdMP87,S95}; for other applications see \cite{LNYan83,T85,PA02} and references).

In this paper we will consider the 
the following abstract version of equation \eqref{e FBSL}:
\begin{gather}
\frac{d \psi}{ dx} = - JL \psi (x) \qquad (0<x<\tau<\infty) .
          \label{e PsiEq}
\end{gather}
Here $L$ and $J$ are operators in an abstract Hilbert space $H$ such that $L$ is a self-adjoint (bounded or unbounded) operator and 
$J$ is a \emph{signature operator} in $H$, that is $J=J^*=J^{-1}$. 

By $P_\pm$ we denote the orthogonal projections onto 
$H_\pm := \ker (J \mp I)$.
Clearly, 
\[
H=H_+ \oplus H_-  \quad \text{and}  \quad J=P_+ \oplus P_-.
\] 
The aim is to find \emph{strong solutions} of the associated boundary problem, i.e., to find continuous functions 
$\psi : [0,\tau] \rightarrow H$ ($\psi \in C ([0,\tau];H)$) which is strongly continuously differentiable on 
$(0,\tau)$ ($\psi \in C^1 ((0,\tau);H)$)
and  satisfies Eq. \eqref{e PsiEq} with the following boundary conditions 
\begin{gather}
P_+ \psi (0) = \vphi_+  , \qquad 
P_- \psi (\tau) = \vphi_-
          \label{e PsiBC} ,
\end{gather}
where $\vphi_+ \in H_+ $ and $\vphi_- \in H_- $ are given vectors. 
If we define $L$ by \eqref{e L} and put 
\begin{gather} \label{e HJ}
H=L^2 (\R, |w(\mu)|d\mu), \qquad
(Jf)(\mu):= (\sgn \mu) f(\mu), 
\end{gather}
we get problem \eqref{e FBSL}-\eqref{e FBBC}.

For the case when $L$ is nonnegative and 
has discrete spectrum, problem \eqref{e PsiEq}--\eqref{e PsiBC} has been described in great detail 
(see \cite{BG68,Pag76,BP83,Pyat84,B85,GvdMP87,KvdMP87,Pyat02} and references therein).
For methods used in these papers, the assumption  
$\sigma (L) = \sigma_{disc} (L)$ \quad 
or the weaker assumption
$\inf \sigma_{ess} (L) >0$ 
is essential. 
Generally, the latter assumption is not fulfilled for Eq. \eqref{e FBSL}.
The simplest example is the equation 
\begin{gather}
(\sgn \mu) |\mu|^\alpha \frac {\partial \psi }{\partial x} (x,\mu ) =  
\frac {\partial^2 \psi}{\partial \mu^2} (x, \mu ) 
\qquad ( 0<x<\tau \leq \infty, \ \mu \in \R ), \quad \alpha>-1, 
\label{e |mu|p}
\end{gather}
which arises in kinetic theory and in the theory of stochastic processes (see \cite{Pag75,GvdMP87,S95} and references in \cite{Pag74}). 
Indeed, for this equation the operator $L= -|\mu|^{-\alpha} \frac{d^2}{d \mu^2}$ is self-adjoint in $L^2 (\R, |\mu|^\alpha d\mu)$ and $\sigma (L) = \sigma_{ess} (L) = [0,+\infty)$.

In the case $\tau < \infty$, problem \eqref{e |mu|p}, \eqref{e FBBC} was studied in \cite{T85} with $\alpha=0$.
In the case $\tau= \infty$ (the half-space problem), one has a boundary condition of the type \eqref{e FBBC} 
at $x=0$ and, in addition, a growth condition on $\psi(x,\mu)$ for large $x$. The half-space problem for Eq. \eqref{e |mu|p} was considered 
in \cite{Pag74,Pag75,G75} in connection with stationary equations of Brownian motion. Note  that the methods of \cite{Pag74,Pag75,G75,T85} use the special form of the weight $w$ and corresponding integral transforms.
The results achieved in \cite{Pag75} for the sample case $\alpha=1$
was used in \cite{Pag76}, where a wider class of problems was considered under the hypotheses that 
\begin{gather}
\text{the weight} \quad w \quad
\text{is bounded} \label{e wbound}
\end{gather}
and $w(\mu) = \mu + o(\mu)$ as $\mu \to 0$.
However, all these results were obtained 
under additional assumptions on the boundary data. In particular, it was supposed that $\vphi_\pm$ are continuous.
 
The case when $L$ may be unbounded and may have a continuous spectrum was considered in \cite[Section 4]{GGvdM88},
where the half-space problem is studied (in an abstract setting) under the assumptions \eqref{e wbound} and $L > \delta >0$.
This assumption was changed to $L>0$ in \cite{C00}. However, it is difficult to apply the results of \cite{C00} to equation \eqref{e |mu|p} since an additional assumption on the boundary values $\phi_\pm$ appears. This assumption is close to assumption 
of \cite{Pag74,Pag75,Pag76}.
The method of \cite{GGvdM88,C00} is based on 
the spectral theory of self-adjoint operators in Krein spaces. 

The aim of this paper is to modify the Krein space approach of \cite{GGvdM88,C00} and to prove that problem \eqref{e PsiEq}-\eqref{e PsiBC} has a unique solution for arbitrary 
$\vphi_\pm  \in H_\pm $.  
In particular, it will be shown that the problem \eqref{e |mu|p}, \eqref{e FBBC} has a unique solution for arbitrary $\vphi_\pm \in L^2 (\R_\pm, |\mu|^\alpha d\mu)$. More general equations of the Fokker-Plank type will be considered also.

Recall that two closed operators $T_1$ and $T_2$ in a Hilbert space $H$ are called similar 
if there exist a bounded and boundedly invertible operator $S$ 
in $H$ such that $S\dom(T_1) = \dom(T_2)$ and 
$T_2 = S T_1 S^{-1}$. 

The central result is the following theorem.

\begin{theorem} \label{t EU} 
Assume that the operator $L$ is self-adjoint and positive 
(i.e., $L \geq 0$ and $\ker L = 0$). Assume that 
the operator $JL$ is similar to a self-adjoint operator in the Hilbert space $H$. 
Then for each pair $\{\vphi_+,\vphi_-\}$, 
$\vphi_\pm  \in H_\pm $, there is a unique strong solution $\psi $ 
of problem \eqref{e PsiEq}-\eqref{e PsiBC}.  
This solution is given by \eqref{e sol}.
\end{theorem}

The proof is given in Subsection \ref{ss prTh1}.
The half-space problem ($\tau = \infty$) is considered in Subsection \ref{ss h-s}. In Section \ref{ss cor n-h} we consider correctness and nonhomogeneous equations.

The formal similarity between Eqs. 
\eqref{e FBSL}-\eqref{e FBBC} and certain 
problems of neutron transport, radiative transfer, and rarefied 
gas dynamics has given rise to the emergence of abstract 
kinetic equation 
\begin{gather}
T \frac{d \psi}{ dx} = - A \psi (x) \qquad (0<x<\tau \leq \infty) ; \label{e AKE} 
\end{gather}
see \cite{H76,B79,vdM81,KLH82,GvdMP87,GGvdM88} and references therein. When Eq. \eqref{e AKE} is considered in a Hilbert space $H$, the operator $T$ is self-adjoint and injective. The operator $A$ is called 
\emph{a collision} operator, usually it satisfies certain positivity assumptions (see e.g. \cite{GvdMP87}). For unbounded collision operators, equation \eqref{e AKE} is usually considered in the space $\Hte_T$ that is a completion of $\dom (T)$ with respect to (w.r.t.) the scalar product
$\btel \cdot ,\cdot \bter_T := (|T| \cdot , \cdot)_{H}$. 
The interplay of $\dom (T)$ and $\dom(A)$ may be various. This leads to additional assumptions on the operators $T$
and $A$. It is assumed in \cite[Section 4]{GGvdM88} that $T$ is bounded and $A>\delta>0$; in \cite{C00}, $\ran(T) \subset \ran (A)$. Note that equation \eqref{e |mu|p} can not be included in these settings.

The second goal of the present paper is to remove the assumptions 
mentioned above. We will show that the following condition is natural for the case when $A$ is unbounded:
\emph{the operator $A$ is a positive self-adjoint operator from $\Hte_T$ to the space $\Hte_T^\prime$} that is  a completion of $\dom (T^{-1})$ w.r.t. the scalar product
$(|T|^{-1} \cdot , \cdot)_{H}$, see Section \ref{s AKE} for details. It is weaker 
than the assumptions mentioned above. On the other hand, it characterizes the case when equation \eqref{e AKE} may be reduced 
to equation \eqref{e PsiEq}. 

Theorem \ref{t EU} leads to the similarity problem for J-positive differential operators.
In Section \ref{s ex}, we use recent results concerning 
the similarity  \cite{CL89,FN98,KarM04,KosMFAT06,KarM06, KosMZ06,KosDis,KarKosM06} (see also \cite{KarMFAT00,KarMN00,FSh02} and references in \cite{KarM06,KarKosM06}) to prove uniqueness and existence theorems for various equations of the type \eqref{e FBSL}.

Note also that abstract kinetic equations with nonsymmetric  collision operators may be found in \cite{KLH82,GvdMP87,GGvdM88,C00,vdMRR00,PA02}. 
From other point of view, equation \eqref{e FBSL} belongs to the class of second order equations with nonnegative characteristic form. Boundary problems for this class of equations were considered by various authors (see \cite{KN67,OR73} and references). But some restrictions imposed in this theory makes it inapplicable to Eq. \eqref{e FBSL} (see
a discussion in \cite{Pag76}).
The case when $w$ is dependent of $\mu$ or the operator $L$ is dependent of $x$ was considered, e.g., in \cite{Aa98,Paron04}.

The main results of this paper were announced in the short communications \cite{KarGAMM06,KarKr06}.

\textbf{Notation.} 
Let $\T$ be a linear operator from a Banach space $H_1$ to a Banach space $H_2$. 
In what follows, $\dom (\T)$,
$\ker \T$, $\ran \T$, $\| \T \|_{H_1 \to H_2}$ are the domain, kernel, range, and norm of $\T$,
respectively. If M is a subset of $H_1$, 
then $\T M := \{ \T h \, : \, h \in M \}$.
In the case $H_1=H_2$, $\sigma(\T)$ and $\rho (\T)$ denote
the spectrum and the resolvent set of $\T$, respectively. As usual,  $\sigma_{\disc} (\T)$ denotes  the discrete spectrum  of $\T$, that is,  the set of isolated eigenvalues of finite algebraic
multiplicity; the essential spectrum is $\sigma _{\ess} (\T):= \sigma (\T) \setminus
\sigma_{\disc} (\T)$.
By $E^\T (\cdot)$ we denote the spectral function of a self-adjoint (or J-self-adjoint) operator $\T$. 
We write $f \in AC_{\loc} (\R)$ if the function $f$ is 
absolutely continuous on each bounded interval in $\R$. 
Put $\R_+:=(0, +\infty)$, $
\R_-:=(-\infty, 0)$, and $\overline \R := \R \cup{ \infty }$.

\section{Existence and uniqueness of solutions\label{s Jpos}}

\subsection{Preliminaries \label{ss Krein}}

In this section basic facts from the  theory of operators in Krein spaces
are collected. The reader can find more
details in \cite{AzJ89,L82}.

Consider a complex Hilbert space $H$ with a scalar product $(\cdot, \cdot
)_H$ and the norm $\| \cdot \|_H := (\cdot, \cdot)^{1/2}$. 
An operator $\T$ in $H$ is called \emph{positive (negative)} if 
$(\T h,h)_H >0$  (resp., $< 0$) for all $h \in \dom (\T)\setminus \{ 0 \}$.
We write $\T>0$ if $\T$ is positive.  


Suppose that $H = H_+  \oplus H_- $, where $H_+$ and $H_-$
are (closed) subspaces of $H$. Denote by $P_\pm$ the orthogonal
projections from $H$ onto $H_\pm$. Let $ J=P_+ - P_-$ and 
$[ \cdot, \cdot ] := (J \cdot ,\cdot )_H $. Then the pair $\K =
(\H , [\cdot, \cdot])$ is called a \emph{Krein space} (see
\cite{AzJ89,L82} for the original definition). The form $[
\cdot, \cdot ]$ is called \emph{an inner product} in the Krein
space $\K$ and the operator $J$ is called \emph{a fundamental
symmetry} (or \emph{a signature operator}) in the Krein space $\K$. 
Evidently, $(\H , [\cdot, \cdot])$ is a Hilbert space if and only if $H_- = 0$.


A subspace $H_1 \subset H$ is called \emph{non-negative (non-positive)} if
$[h,h] \geq 0$ ($\leq 0$, resp.) for all $h \in H_1$.
A non-negative (non-positive) subspace $H_1$ is  \emph{uniformly positive (uniformly negative, resp.)} 
if there is a constant $\alpha>0$ such that $\alpha \| h \|_H \leq [h,h] \ $ ($\leq -[h,h]$, resp.) for $h \in H_1$.
A non-negative (non-positive) subspace $H_1 $ is called maximal non-negative (non-positive)
if for any non-negative (non-positive) subspace $H_2 \supset H_1$ we have $H_2 = H_1$. 
Subspaces $H_1$ and $H_2$ are called J-orthogonal if $[h_1,h_2]=0$ for all  $h_1 \in \H_1$, $h_2 \in \H_2$. 
We write $H_1= H_2 [\dot +] H_3$ if $H_1$ admits decomposition into the direct sum of two J-orthogonal subspaces $H_2$ and $H_3$. 


Suppose that subspaces $\H_\pm$ possess the properties
\begin{description} 
\item[(i)] $\H_+$ is non-negative, $\H_-$ is non-positive,
\item[(ii)] $ (\H_+, [\cdot,\cdot])$ and $(\H_-, -[\cdot,\cdot])$ are Hilbert spaces;
\end{description}
and suppose that $ H= \H_+ [\dot +] \H_-$; 
then this decomposition is called a \emph{canonical decomposition} 
of the Krein space $\K$.

Evidently, $H=H_+ \oplus H_-$ is a canonical decomposition. 
Note that there exist other canonical decompositions (see Proposition \ref{p cr=sim}).

Let $H=\H_+ \oplus \H_-$ be a canonical decomposition. Then the norm  
$\| \cdot \|_{\H_\pm} := \sqrt {\pm [\cdot,\cdot]} $ in the Hilbert space $(\H_\pm, \pm [\cdot,\cdot])$
is called \emph{an intrinsic norm}. 
It is easy to prove (see  \cite[Theorem I.7.19]{AzJ89}) that the norms $\| \cdot \|_{\H_\pm}$ and 
$\| \cdot \|_{H}$ are equivalent on $\H_\pm$; moreover, 
\begin{gather} \label{e EqNorms}
 \gamma \| h_\pm \|_H  \leq \| h_\pm \|_{\H_\pm} \leq \| h_\pm \|_H, \quad h_\pm \in \H_\pm, 
\end{gather} 
where $ \gamma \in (0,1] $ is a constant.

Statement (i) of the following proposition is due to Ginzburg (see 
\cite[Theorems I.4.1 and I.4.5]{AzJ89}), Statement (ii) is due to 
Phillips and Ginzburg
(see \cite[Lemma I.8.1]{AzJ89}). 

\begin{proposition}[e.g. \cite{AzJ89}] \label{p proj}
Let $H=\H_+ \oplus \H_-$ be a canonical decomposition and let $\pr_+$ and $\pr_-$ be corresponding 
mutually complementary projections on $\H_+$ and $\H_-$, respectively.
\begin{description} 
\item[(i)] If $H_1$ is a maximal non-negative subspace in $H$, then the restriction 
$ \pr_+ \upharpoonright H_1 \, : \, H_1 \rightarrow \H_+ $ is a homeomorphism,
that is, it is bijective, continuous, and the inverse mapping 
$(\pr_+ \upharpoonright H_1)^{-1}: \H_+ \to H_1$ is also continuous.
\item[(ii)] If $H_1$ is a uniformly positive subspace in $H$, then there is a constant 
$\beta \in (0,1)$ such that $ |[\pr_- h , \pr_- h]| \leq \beta [\pr_+ h , \pr_+ h]$.
\end{description}
\end{proposition}

Let $\T$ be a densely defined operator in $H$. The
\emph{J-adjoint operator of} $\T$ is defined by the relation
\[
 \ [\T f,g] =
[f,\T^{[*]}g] \ , \qquad  f \in \dom (\T),
\]
on the set of all $g \in H$ such that the mapping $f \mapsto
[\T f,g]$ is a continuous linear functional on $\dom (\T)$. The
operator $\T$ is called \emph{J-self-adjoint} if $ \T = \T^{[*]} $.
It is easy to see that $\T^{[*]} := J \T^* J$ and the operator $\T$
is J-self-adjoint if and only if $J \T$ is self-adjoint. Note that
$J=J^*=J^{-1}=J^{[*]} $. A closed operator $\T$ is called
\emph{J-positive} if $\ [\T f,f] > 0 \ $ for 
$\ f \in \dom (\T)\setminus \{0\} $ 
(it is equivalent to $J \T > 0$). 

Let $\mathfrak{S}$ be the semiring consisting of all bounded
intervals with endpoints different from $0$ and
their complements in $\overline \R := \R \cup{ \infty }$.


Let $\T$ be a J-positive J-self-adjoint operator in
$H$ with a nonempty resolvent set, $\rho (\T) \neq
\emptyset$. Then $\T$ admits a spectral function $E(\Delta)$. Namely,
\begin{description}
\item[(i)] the spectrum of $\T$ is real, $\sigma(\T) \subset \R$;
\item[(ii)] there exist a mapping $\Delta \rightarrow E(\Delta)$
from $\mathfrak{S}$ into the set of bounded linear operators in
$H$ with the following properties ($\Delta ,\Delta' \in
\mathfrak{S}$):
\begin{description}
\item[(E1)] $E(\Delta \cap \Delta' ) = E(\Delta) E(\Delta' ) $,
\quad $E(\emptyset ) = 0, \quad E(\overline{\R}) = I $, \quad
$E(\Delta) = E(\Delta)^{[*]}$; \item[(E2)] $E(\Delta \cup \Delta'
) = E(\Delta) + E(\Delta' )$ \quad if \quad $\Delta \cap \Delta' =
\emptyset$; \item[(E3)] the form $ \pm [ \cdot,\cdot] $ is
positive definite on $E (\Delta) H $, if $\Delta
\subset \R_\pm$; \item[(E4)] $E(\Delta) $ is in the double
commutant of the resolvent $R_\T(\lambda)=(\T-\lambda)^{-1}$ and
$\sigma (\T \upharpoonright E(\Delta)H) \subset
\overline{\Delta}$; \item[(E5)] if $\Delta$ is bounded, then
$E(\Delta) H \subset \dom(\T)$ and $\T\upharpoonright E
(\Delta) H $ is a bounded operator.
\end{description}
\end{description}

According to \cite[Proposition II.4.2]{L82}, a number $s \in
\{0,\infty\}$ is called \emph{a critical point} of $\T$, if for
each $\Delta \in \mathfrak{S}$ such that $s\in \Delta$, the
form $[\cdot,\cdot ]$ is indefinite on $E(\Delta)\mathfrak{H}$ (i.e., there exist 
$h_\pm \in E(\Delta)\mathfrak{H}$ such that $[h_+,h_+ ]<0$ and $[h_-,h_- ]>0$ ). The set of critical points is denoted by $c(\T)$.

If $\alpha \not \in c (\T)$, then for arbitrary $\lambda_0, \lambda_1
\in \R \setminus c(\T)$, $\lambda_0 < \alpha$, $\lambda_1>\alpha$,
the limits
\begin{equation*} 
\lim_{\lambda \uparrow \alpha} E([\lambda_0,\lambda]) , \qquad
\lim_{\lambda \downarrow \alpha} E([\lambda,\lambda_1])
\end{equation*}
exist in the strong operator topology. 
If $\alpha \in c(\T)$ and the above limits 
do still
exist, then the critical point $\alpha$ is called 
\emph{regular}, otherwise it is called \emph{singular}. Here we agree
that, if $\alpha = \infty$, then $\lambda > \alpha$, $\lambda < \alpha$, $\lambda \downarrow \alpha$, and $\lambda \uparrow \alpha$  mean $\lambda > -\infty$, $\lambda < +\infty$, $\lambda \downarrow -\infty$, and $\lambda \uparrow +\infty$, respectively.

The mapping $\Delta \rightarrow E(\Delta)$ can be extended to the semiring generated by 
those intervals whose endpoints are not singular critical points of $\T$. 
For this extension Properties (E1)-(E5)  are preserved.


\subsection{Proof of Theorem \ref{t EU}. \label{ss prTh1}}

Let $L$ be a positive operator in a Hilbert space $H$ and let $J$ be a signature operator in $H$. Put 
\[
B:=JL.
\]  
Then $B$ is a J-self-adjoint and J-positive operator in the Krein space 
$\K:=(H, [\cdot,\cdot ])$, where $[\cdot,\cdot ] := (J \cdot, \cdot )_H$. 
The following proposition is well known.

\begin{proposition}
\label{p cr=sim}
Let $B$ be a J-positive J-self-adjoint operator in $H$ such that $\rho(B) \neq \emptyset$.
Then the following assertions are
equivalent:
\begin{description}
\item[(i)] $B$ is similar to a self-adjoint operator. 
\item[(ii)] $0$ and $\infty$ are not singular critical points of $B$.
\item[(iii)] There exist a fundamental decomposition $H = H^B_+ [\dot +] H^B_-$ of the Krein space $\K$
such that $B = B^+ \dot + B^-$, where $B^\pm := B \uph \dom(B) \cap H^B_\pm$,
and $\sigma (B^\pm) = \overline{\sigma (B) \cap \R_\pm}$.
\end{description}
If these assertions hold and $P^B_+ $ and $P^B_-$ are corresponding 
mutually complementary projections on $H^B_+$ and $H^B_-$, respectively, then 
$P^{B}_\pm = E^B (\R_\pm)$, where $E^B (\cdot) $ is a spectral function of $B$ defined in Subsection \ref{ss Krein}. 
\end{proposition}

Assume that 
\begin{equation} \label{a sim}
B \quad \text{is similar to a self-adjoint operator in} \ H. 
\end{equation}
Then $\rho (B) \subset \R$ and assertions (i)-(iii) of Proposition \ref{p cr=sim} hold true. Note that Property (E3) implies that the form $ \pm [ \cdot,\cdot] $ is
non-negative on $H^B_\pm$.
Evidently, 
\begin{gather} \label{e HBpm max}
H^B_+ \quad (H^B_-) \quad \text{is a maximal non-negative (non-positive) subspace in}
\quad  H.
\end{gather}
Moreover, it follows from \eqref{e EqNorms} that $H^B_+$ ($H^B_-$) is uniformly positive (resp., negative).
In the sequel, we consider $H^B_\pm$ as Hilbert spaces with the inner products $\pm[\cdot,\cdot]$ 
and the (intrinsic) 
norms $\ \| \cdot \|_{H^B_\pm}:= \sqrt{\pm [h , h ]}$.

\begin{proposition} \label{p BpmS-A}
Let $B$ be similar to a self-adjoint operator. Then 
the operator $ B^+ $ ($B^-$) defined in Proposition \ref{p cr=sim} is 
a positive (negative, resp.) self-adjoint operator in the Hilbert space
$H^B_+$ (resp., $H^B_-$).
\end{proposition}

\begin{proof}
Positivity of $B^+$ follows immediately from J-positivity of $B$. 
$B^+$ is symmetric since it is positive. It follows from  \eqref{a sim} that the spectrum of $B$ is real, 
$\sigma (B) \subset \R$. Thus, $\sigma (B^+) \subset \R$ and therefore $B^+$ is self-adjoint in $H^B_+$. The same arguments hold for $B^-$.
\end{proof}

Now we write equation \eqref{e PsiEq} in the form 

\begin{gather}
\frac{d \psi}{ dx} = - B \psi (x) \qquad (0<x<\tau<\infty) , 
          \label{e BEq} 
\end{gather}
and suppose that $\psi$ is a solution of 
\eqref{e BEq}, \eqref{e PsiBC}.

We put $\psi_\pm (x) := P^B_\pm \psi (x)$ and $E^B_t := E^B ((-\infty, t])$. 
Note that, by Proposition \ref{p cr=sim}, $E^B_t$ is 
well-defined for all $t \in \R$ and 
$E^B (\cdot ) \uph H^B_\pm$ is the spectral function 
of the self-adjoint operator $B^\pm$. 
Eq. \eqref{e BEq} 
implies that for  $x \in [0,\tau]$,
\begin{gather} 
\label{e Psi+(x)}
\psi_+ (x) = \int_{+0}^{+\infty} e^{-xt} \; d E^B_t \ \psi_+ (0) = e^{-x B^+ } \psi_+ (0) , \\ \quad 
\psi_- (x) = \int_{-\infty}^{-0} e^{(\tau-x)t} \; d E^B_t \ \psi_- (\tau) 
= e^{(\tau - x)B^-} \psi_- (\tau) .
\label{e Psi-(x)}
\end{gather}
The integrals converge in the norm topologies of $H^B_\pm$ as well as in the norm topology of $H$.
Recall that, by \eqref{e EqNorms}, these topologies are equivalent.  

It follows immediately from Proposition \ref{p BpmS-A} that for all $h \in H^B_\pm$,
\begin{gather} \label{e exp<1}
 \| e^{\mp x B^\pm} h \|_{H^B_\pm} \leq \| h \|_{H^B_\pm} \qquad \text{if} \quad x \geq 0 .
\end{gather} 
So for each pair $\wt \psi_+ \in H^B_+$, $\wt \psi_- \in H^B_-$, there is a unique solution 
of \eqref{e BEq} such that $\psi_+ (0) = \wt \psi_+$, $\psi_- (\tau) = \wt \psi_-$. 
      
The boundary conditions \eqref{e PsiBC} becomes
\begin{gather} \label{e BC1}
P_+ [\psi_+ (0)+ e^{\tau B^-} \psi_- (\tau)] = \vphi_+ , \quad 
P_- [\psi_- (\tau) + e^{-\tau B^+} \psi_+ (0)] = \vphi_- .
\end{gather}

It follows from \eqref{e HBpm max} and Proposition \ref{p proj} (i) that there exist operators 
\begin{equation} \label{e Rpm}
R_\pm := (P_\pm \upharpoonright H^B_\pm)^{-1} \ : \ H_\pm \rightarrow H_\pm^B \, ; \qquad \text{moreover,} \quad
 R_\pm \quad \text{ are homeomorphisms.}
\end{equation}

Let us introduce operators $G_\pm : H^B_\pm \to H^B_\mp$ by 
\[
G_{+} h_+  := R_- P_- e^{-\tau B^+} h_+ , \quad
G_{-} h_-  := R_+ P_+ e^{\tau B^-} h_-, \quad h_\pm \in H^B_\pm .
\]

Using the operators $R_\pm$ and $G_\pm$, we write \eqref{e BC1} in the form
\begin{gather} \label{e BC2}
\psi_+ (0)+ G_- \psi_- (\tau) = R_+ \vphi_+ , \quad 
\psi_- (\tau) + G_+ \psi_+ (0) = R_- \vphi_- .
\end{gather}
Combining these equations, one gets 
\begin{gather} 
(I - G_- G_+ ) \psi_+ (0)  
= R_+ \vphi_+ - G_- R_- \vphi_- , 
\qquad  
(I - G_+ G_- ) \psi_- (\tau) = 
R_- \vphi_- - G_+ R_+ \vphi_+ .\label{e BC3}
\end{gather}

\begin{lemma} \label{l G}
\quad $ \| G^+ \|_{H^B_+ \to H^B_-} <1 $ \quad and \quad $\| G^- \|_{H^B_- \to H^B_+} <1 $. 
\end{lemma}

\begin{proof}
Let us prove that for $h_\pm \in H^B_\pm$,
\begin{gather} \label{e R_+<1}
 \| R_\mp P_\mp h_\pm \|_{H^B_\mp} \leq \beta_\pm \| h_\pm \|_{H^B_\pm} \qquad  \text{with certain constants} \quad \beta_\pm <1 .
\end{gather} 
Indeed, since $R_- P_- h_+ \in H^B_-$, we have $h_+ = P^B_+ (h_+ - R_- P_- h_+)$ and therefore 
\begin{gather} \label{e h+norm=}
\| h_+ \|_{H^B_+} = \| P^B_+ (h_+ - R_- P_- h_+) \|_{H^B_+} . 
\end{gather}
Note that $H = H^B_+ [\dot +] H^B_-$ is a fundamental decomposition of 
the Krein space $\K$ and $H_+$ is a uniformly positive subspace in $\K$.
Hence  Proposition \ref{p proj} (ii) implies 
\begin{gather} \label{e Pnonneg}
\beta_+ \| P^B_+ g_+ \|_{H^B_+} \geq \| P^B_- g_+ \|_{H^B_-} 
\qquad \text{for all} \quad g_+ \in H_+
\end{gather}
with  a certain $\beta_+ < 1$.
Further,  $P_- (h_+ - R_- P_- h_+) = 0$, that is, $(h_+ - R_- P_- h_+ ) \in H_+$.
Therefore \eqref{e Pnonneg} yields 
\[
\beta_+ \| P^B_+ (h_+ - R_- P_- h_+) \|_{H^B_+} \geq \| P^B_- (h_+ - R_- P_- h_+ )\|_{H^B_-} = 
\| R_- P_- h_+ \|_{H^B_-} .
\]
From this and \eqref{e h+norm=}, we get \eqref{e R_+<1} for $h_+ \in H^B_+$.  The proof of \eqref{e R_+<1} for $h_- \in H^B_-$ is the same. 
To conclude the proof, it remains to combine \eqref{e R_+<1} and \eqref{e exp<1}.
\end{proof}

Lemma \ref{l G} implies that 
\begin{gather} \label{e I-GG}
I-G_\mp G_\pm \qquad \text{are linear homeomorphisms in} \quad H^B_\pm.
\end{gather}

Let us consider Eqs. \eqref{e BC1} as a system for the unknowns  
$\psi_+ (0) $ and $ \psi_- (\tau)$.

\begin{lemma}[\cite{C00}] \label{l sys}
System \eqref{e BC1} has a unique solution 
\begin{gather} 
\psi_+ (0) = (I - G_- G_+)^{-1} ( R_+ \vphi_+ - G_- R_- \vphi_- ), \label{e BC4i1} \\ 
\psi_- (\tau) = (I - G_+ G_-)^{-1} (R_- \vphi_- - G_+ R_+ \vphi_+ ).\label{e BC4i2}
\end{gather}
\end{lemma}

\begin{remark}
Lemma \ref{l sys} was obtained in other form in \cite{C00}
(see Lemma 2.2 and the end of the proof of Theorem 3.4 there).
Earlier, it was proved under the additional condition 
$\sigma(B)=\sigma_{\disc} (B)$ in \cite{BP83}, see also \cite{Pyat84,GvdMP87}.
We give another proof, which is based on Lemma \ref{l G} and is an improvement of treatments from \cite{Pyat84}.
\end{remark}

\begin{proof}
First note that system \eqref{e BC1} is equivalent to system \eqref{e BC2}.   
Further, \eqref{e BC2} implies \eqref{e BC3}. On the other hand, 
\eqref{e BC3} is equivalent to \eqref{e BC4i1}-\eqref{e BC4i2} due to \eqref{e I-GG}.  
Thus we should only prove that \eqref{e BC3} implies \eqref{e BC2}.

Let us write Eqs. \eqref{e BC3} as
\begin{gather*}
R_+ \vphi_+ = 
(I - G_- G_+ ) \psi_+ (0) + G_- R_- \vphi_-   ,  \qquad
R_- \vphi_- = (I - G_+ G_- ) \psi_- (\tau) + G_+ R_+ \vphi_+ . 
\end{gather*}
Combining these two equalities, we get
\begin{gather*}
(I - G_- G_+ ) R_+ \vphi_+ = (I - G_- G_+ ) \psi_+ (0) +  G_- (I - G_+ G_- ) \psi_- (\tau) ,  \\ 
(I - G_+ G_- ) R_- \vphi_- = (I - G_+ G_- ) \psi_- (\tau) + G_+ (I - G_- G_+)  \psi_+ (0).  
\end{gather*}
Note that $G_\pm (I - G_\mp G_\pm ) = (I - G_\pm G_\mp )G_\pm $
and therefore, using \eqref{e I-GG}, one obtains \eqref{e BC2}.
\end{proof}

This lemma shows that the function 
\begin{gather} \label{e sol}
\psi (x) = e^{-t B^+ } (I - G_- G_+)^{-1} ( R_+ \vphi_+ - G_- R_- \vphi_- ) 
+ e^{(\tau - t)B_-} (I - G_+ G_-)^{-1} (R_- \vphi_- - G_+ R_+ \vphi_+ ) 
\end{gather}
is a unique solution of problem \eqref{e PsiEq}-\eqref{e PsiBC}.

One can check that the functions $\psi_\pm$ are continuous on $[0,\tau]$, the strong derivatives $d \psi_\pm /dx$ exist and are (strongly) continuous on $(0,\tau)$ (see Subsection~\ref{ss  cor n-h}).
This completes the proof of Theorem~\ref{t EU}.

\subsection{Half-space problems \label{ss h-s}} 

Under the same assumptions, let us consider the equation  
\begin{gather}
\frac{d \psi}{ dx} = - JL \psi (x) \qquad (0<x<\infty) ,
          \label{e EqInf} 
\end{gather}
on the infinite interval $(0,+\infty)$.          
The boundary conditions 
\begin{gather}
P_+ \psi (0) = \vphi_+  , \qquad \label{e BCat0} \\ 
\| \psi (x) \|_{H} = O(1) \qquad (x \rightarrow +\infty)   
          \label{e BCatInf}
\end{gather}
corresponds to this feature of the problem (see e.g. \cite{GvdMP87}). 
As above, $\vphi_+ \in H_+$ is a given vector.

\begin{theorem} \label{t EU Jpos} 
Assume that $L =L^* > 0$ and that the operator $B:=JL$ is similar to a self-adjoint operator in the Hilbert space $H$. 
Then for each $\vphi_+  \in H_+ $ there is a unique solution $\psi $ of 
\eqref{e EqInf}-\eqref{e BCat0}-\eqref{e BCatInf}.
This solution is given by 
\begin{gather}
\psi (x) = \int_{+0}^{+\infty} e^{-xt} \; d E^B_t R_+ \vphi_+ \quad (= e^{-x B^+ } R_+ \vphi_+),
\label{e PsiInf}
\end{gather} 
where $E_t^B$, $R_+$, and $B^+$ are the operators defined in Subsection \ref{ss prTh1}.
\end{theorem} 

The proof is simpler than the proof of Theorem \ref{t EU}. 
It is similar to the treatments from \cite[Section 4]{GGvdM88},
where equation \eqref{e AKE} with bounded $T$ and uniformly positive $A$ was considered. 
We give a sketch here.

\begin{proof}
Let $\psi $ be a solution of problem \eqref{e EqInf}-\eqref{e BCatInf} and   
$\psi_\pm (\cdot) := P^B_\pm \psi (\cdot)$.
It follows from \eqref{e Psi-(x)} that for all $x \in (0,+\infty)$,
\begin{gather*} 
\psi_- (0) \in \ran (e^{xB^-}) = \dom (e^{-xB^-}) \qquad \text{and} \quad \psi_- (x) = e^{-xB^-} \psi_- (0) .
\end{gather*}
Since $B^- <0$, we see that 
\[
 \lim_{x \to +\infty} \| \psi_- (x) \|_{H^B_-} = \infty \qquad \text{whenever} \quad 
\psi_- (0) \neq 0. 
\]
Taking into account \eqref{e EqNorms} and \eqref{e BCatInf}, we get 
$\psi_- (0) = \psi_- (x) = 0$ for all $x \in (0,\infty)$. 
Hence, $\psi (x) = \psi_+ (x)$, $x \geq 0$, and therefore \eqref{e BCat0} yields  $P_+ \psi_+ (0) = \vphi_+$. Combining this with \eqref{e BCat0} and \eqref{e Rpm}, 
one can see that $\psi (x) = \psi_+ (x) = e^{-xB^+} R_+ \vphi_+ $ is an only function that satisfies \eqref{e EqInf} and \eqref{e BCat0}. Finally, note that \eqref{e BCatInf}
follows from \eqref{e exp<1} and \eqref{e EqNorms}.
\end{proof}

\begin{remark}
Clearly, Theorem \ref{t EU Jpos} is valid with the condition 
$\| \psi (x) \|_{H} = o(1)$ as $x \rightarrow +\infty$ instead of 
\eqref{e BCatInf}.
\end{remark}

\section{Correctness and nonhomogeneous problems \label{ss cor n-h}}

Let the assumptions of Subsection \ref{ss prTh1} be fulfilled.

Since $B^+$ is a positive self-adjoint operator in $H^B_+$,
we see that $U_+ (z) := e^{-zB^+}$ is a bounded holomorphic semigroup in the sector 
$|\arg z|< \pi/2$ (see e.g. \cite[Subsection IX.1.6]{Kato66}). 
The same is true for the function $U_- (z) := e^{zB^-}$.
In particular, this implies that for any $\wt \psi_\pm \in H^B_\pm$ problems 
\begin{gather} \label{e pr+}
d\psi_+ (x) /dx = -B^+ \psi_+ (x), \quad 
x >0, \qquad \psi_+ (0) = \wt \psi_+, 
\end{gather}
\begin{gather} \label{e pr-}
d\psi_- (x) /dx = -B^- \psi_- (x), \quad 
x < \tau , \qquad \psi_- (\tau) = \wt \psi_-, 
\end{gather} 
 have unique solutions $\psi_+ (x) = U_+ (x) \wt \psi_+$ and $\psi_- (x) = U_- (\tau-x) \wt \psi_-$, respectively. 
Besides, these solutions are infinitely differentiable on $(0,\tau)$ and  problems \eqref{e pr+}, \eqref{e pr-} are uniformly correct (see e.g. \cite[Subsection I.2]{KSG71}).
The letter means that 
the mappings $ \wt \psi_\pm \to \psi_\pm (\cdot)$ from $H^B_\pm $ to $C([0,\tau];H^B_\pm)$ 
are continuous. 

Lemma \ref{l sys} and \eqref{e Psi+(x)}-\eqref{e Psi-(x)} show that 
the solution $\psi $ of problem \eqref{e PsiEq}-\eqref{e PsiBC} possesses the same properties:
\begin{description}
\item[(i)] $\psi $ are infinitely differentiable on $(0,\tau)$,
\item[(ii)] the mapping $ \{\vphi_+,\vphi_- \} \to \psi (\cdot)$ from $H_+ \oplus H_- $ to $C([0,\tau];H)$ 
is continuous. 
\end{description}
One can obtain similar statements for the solution of problem 
\eqref{e EqInf}-\eqref{e BCatInf}.

Now consider the nonhomogeneous equation 
\begin{gather}
\frac{d \psi}{ dx} = - JL \psi (x) + f(x) \qquad (0<x< \tau \leq \infty) ,
          \label{e NonHomEq} 
\end{gather}
where $f$ is an $H$-valued function. 

We assume that $f$ is H\"older continuous on 
all finite intervals $[0, x_1 ]$, i.e., 
\begin{gather} 
\text{for each} \quad x_1 \in \R_+ \quad \text{there are numbers} \quad K=K(x_1)>0, \quad k=k (x_1) \in (0,1] \quad 
\text{such that} \notag \\ 
\| f(x) - f(y) \|_{H} \leq K |x-y |^{k} \qquad \text{for} \quad 0 \leq x,y \leq x_1 \label{e asHol}.
\end{gather}
Evidently, the functions $f_\pm (x) := P^B_\pm f (x)$ possess the same property.

Let us start from the case $ \tau < \infty$ and boundary conditions \eqref{e PsiBC}.

The fact that $U_\pm (z)$ are bounded holomorphic semigroups enables us to apply \cite[Theorem IX.1.27]{Kato66}.
This theorem yields that 
the functions 
\begin{gather} \label{e psi1pm}
\psi_1^+ (x) := \int_0^x U_+ (x-y) f_+ (y) dy  \quad \text{and}  \quad
\psi_1^- (x) := -\int_x^\tau U_- (y-x) f_- (y) dy 
\end{gather}
are continuous for $x \in [0,\tau]$, continuously differentiable for $x \in (0,\tau)$ and 
$d \psi_1^\pm / dx = - B^\pm \psi_1^\pm + f_\pm$. 
By Theorem \ref{t EU}, there exist a unique solution $\psi_0 $ of homogeneous equation
\eqref{e BEq} satisfying the boundary conditions 
\begin{gather*}
P_+ \psi_0 (0) = \vphi_+ - P_+ \psi_1^- (0), \qquad 
P_- \psi_0 (\tau) = \vphi_- - P_- \psi_1^+ (\tau) .
\end{gather*}
The representation of $\psi_0$ may be obtained from \eqref{e sol}.
Thus we prove the following statement.

\begin{proposition} \label{p NHfin}
Let assumptions \eqref{a sim} and \eqref{e asHol} be fulfilled.
Then problem \eqref{e NonHomEq}, \eqref{e PsiBC} has a unique  solution $\psi$ given by 
$\psi = \psi_1^+ + \psi_1^- + \psi_0 $.
\end{proposition}

In the case $\tau = \infty$, we assume additionally that 
\begin{gather} \label{e fL1}
\int_0^{+\infty} \| f(x) \|_H dx < \infty .  
\end{gather}

Let us define $\psi_1^\pm (\cdot)$ by \eqref{e psi1pm} (with 
$\tau = +\infty$). Assumption \eqref{e fL1} and inequality \eqref{e exp<1} yield that $\psi_1^- (\cdot)$ is well-defined on $[0,\infty)$. 

\begin{proposition} \label{p NHinf}
Let $\tau = \infty$, let assumptions \eqref{a sim}, \eqref{e asHol}, and \eqref{e fL1} be fulfilled.
Then problem \eqref{e NonHomEq}, \eqref{e BCat0}, \eqref{e BCatInf} has a unique  solution $\psi$ given by 
$\psi = \psi_1^+ + \psi_1^- + \psi_0 $,
where 
\begin{gather*} 
\psi_0 (x) = e^{-x B^+ } R_+ (\vphi_+ - P_+ \psi_1^- (0)) .
\end{gather*}
\end{proposition}

\begin{proof}
For $0<x<X$, we write $\psi_1^- (x)$ in the form $v_1 (x) + v_2 (x)$, where
\[
v_1 (x) =  -\int_x^X U_- (y-x) f_- (y) dy, \quad
v_2 (x) =  -\int_X^{+\infty} U_- (y-x) f_- (y) dy.
\]
Since $v_2 (x) =  - U_- (X-x) \int_{X}^{+\infty} U_- (y-X)  f_- (y) dy$, we see that $dv_2 (x)/dx = -B^- v_2 (x)$
for $0<x<X$. On the other hand, \cite[Theorem IX.1.27]{Kato66} yields that $v_1$ is a solution of nonhomogeneous equation 
$d\psi^- (x)/dx = - B^- \psi^- (x) + f_- (x)$ for 
$0<x<X$ and, therefore, so is $\psi_1^- $. The latter holds for  all $x>0$ since $X$ is arbitrary.
Thus $\psi_1^+ + \psi_1^- $ is a solution of nonhomogeneous equation \eqref{e NonHomEq}. 

It follows from \eqref{e fL1} and \eqref{e exp<1} that 
$\| \psi_1^+ (x) \|= O(1)$ and $ \| \psi_1^- (x) \| = o(1)$ as $x \to \infty$. 
Using Theorem \ref{t EU Jpos}, one can conclude the proof.
\end{proof}

If $L\geq\la_0>0$, then condition \eqref{e fL1} can be relaxed. 
One can change it to $ \| f(x) \| = O(1)$ as $x \to \infty$ or to 
$\int_0^\infty e^{\la_0 x} \| f(x) \| dx < \infty$ (cf. \cite[Section II]{GvdMP87}).

\section{Abstract kinetic equations \label{s AKE}}

Let $\Hte$ be a complex Hilbert space with scalar product $ \btel \cdot , \cdot \bter$ 
and norm $\| \cdot \|_\Hte$. 

Assume that $T$ is a (bounded or unbounded) self-adjoint operator in $\Hte$ and that
T is injective (i.e., $\ker T =0$).

Let $Q_+ := E^{T} (\R_+)$  $(Q_- := E^{T} (\R_-))$ 
be the orthogonal projection of $\Hte$ 
onto the maximal T-invariant subspace on which $T$ is positive (negative).
Then 
$|T|:=(Q_+ - Q_-) T$ is a positive self-adjoint operator.
Note that
\begin{gather} \label{e QT=TQ}
 Q_\pm T = T Q_\pm \qquad \text{and}  \qquad Q_\pm T^{-1} = T^{-1} Q_\pm .
\end{gather}

Following \cite{B79}, let us introduce the scalar product 
$\btel  h , g \bter_T = \btel |T| h , g \bter$ for $h,g \in \dom (T)$ 
with corresponding norm $\| \cdot \|_T $ and denote by $\Hte_T$ 
the completion of $\dom (T)$ with respect to (w.r.t.) this norm. Clearly, 
$\Hte \cap \Hte_T = \dom (|T|^{1/2})$ and $\| h \|_T = \| \, |T|^{1/2} h \, \|_\Hte$
for $h \in \dom (|T|^{1/2})$. 

Similarly, we may introduce another scalar product 
$\btel  h , g \bter_T^\prime = \btel |T|^{-1} h , g \bter_T$ 
on $\dom (T^{-1}) $ with associated norm $\| \cdot \|_T^\prime $ and consider 
the completion $\Hte_T^\prime$ of $\dom (T^{-1})$ w.r.t. the norm $\| \cdot \|_T^\prime $.
As before, $\Hte \cap \Hte_T^\prime = \dom (|T|^{-1/2})$. 

It is easy to see that for each $g \in \ran (|T|) = \dom (T^{-1})$, the linear functional 
$\btel \cdot, g \bter$ is continuous on $\dom (T)$ w.r.t. the norm $\| \cdot \|_T$.
Besides, its norm is equal to 
\[
\sup_{\substack{h \in \dom (T) \\ h \neq 0}} \frac {|\btel h,g \bter|}{ \| \, |T|^{1/2} h \, \|_\Hte}
= \sup_{\substack{h' \in \dom (|T|^{1/2}) \\ h' \neq 0}} = 
\frac {|\btel |T|^{-1/2} h', g \bter|}{ \| \, h' \, \|_\Hte} 
= \| \, |T|^{-1/2} g \, \|_{\Hte} = \| g \|_T^\prime .
\]
So one can use the $\Hte$-scalar product as a pairing to identify 
$\Hte_T^\prime$ with the dual $\Hte_T^*$ of $\Hte_T$. 
The operator $|T|$ ($|T|^{-1}$) has a natural isometric extension from $\Hte_T$ onto $\Hte_T^\prime$ 
(from $\Hte_T^\prime$ onto $\Hte_T$). We use the same notation for the extensions.

By \eqref{e QT=TQ},
we may extend 
the orthogonal projection $Q_\pm$ onto $\Hte_T$ and $\Hte_T^\prime$. 
Put $T h := (Q_+ - Q_-) |T| h$ for $h \in \Hte_T$ and  
$T^{- 1} h := (Q_+ - Q_-) |T|^{-1} h$ for $h \in \Hte_T^\prime$. 


Now let $A$ be a linear operator in $\Hte + \Hte_T + \Hte_T^\prime$, 
let $\vphi$ be a vector in $\Hte_T$, and let
$f (\cdot)$ be a function with values in $\Hte_T^\prime$.
Consider an abstract kinetic equation 
\begin{gather}
T \frac{d \psi}{ dx} = - A \psi (x) + f(x)\qquad (0<x<\tau \leq \infty) ,
          \label{e TE} 
\end{gather}
supplemented by "half-range" boundary conditions in the form
\begin{gather}
Q_+ \psi (0) = Q_+ \vphi  , \qquad \label{e TE BC0} \\ 
Q_- \psi (\tau) = Q_- \vphi  \quad \text{if} \quad \tau < \infty,  \qquad \text{or}
\quad \| \psi (x) \|_{\Hte_T} = O(1) \qquad (x \rightarrow +\infty) \quad \text{if} \quad \tau=\infty .
        \label{e TE BCtau}
\end{gather}

In the abstract kinetic theory the operator $A$ is called \emph{a collision operator}. 
It may have any of a number of properties (see \cite{KLH82,GvdMP87}).

Here we consider 
the case when \emph{$A$ is a positive self-adjoint operator from 
$\Hte_T$ to $\Hte_T^\prime$}. 
That is, 
\begin{description}
\item[(A1)] $\dom (A) \subset \Hte_T$, and $Ah \in \Hte_T^\prime$ for all $h \in \dom (A)$,
\item[(A2)] $\btel A h , h \bter > 0$ for all $ h \in \dom (A) \setminus \{ 0\}$,
\item[(A3)] $A=A^*$, i.e., $\dom (A)$ coincides with the set of all $g \in \Hte_T$ such that the mapping $h \mapsto
\btel Ah,g \bter $ is a continuous linear functional w.r.t. the $\Hte_T$-norm.
\end{description}

We seek \emph{$\Hte_T$-strong solutions} (week solutions in terms of \cite[Section 2]{GvdMP87}) of problem \eqref{e TE}-\eqref{e TE BCtau}.
That is, it is supposed that $d/dx$ is the strong derivative in $\Hte_T$ and that
\begin{description}
\item[(i)] $\psi ( \cdot ) \in C([0,\tau]; \Hte_T)$ (or $\psi ( \cdot ) \in C([0,+\infty); \Hte_T) $ if $\tau = \infty$),
\item[(ii)] $\psi ( \cdot ) \in C^1 ((0,\tau); \Hte_T )$ and 
$\psi (x) \in \dom (A)$ for $x \in (0,\tau)$. 
\end{description}

\begin{theorem} \label{t TEFin}
Let $B$ be the operator in $\Hte_T$ defined by 
$B :=T^{-1}A $, $\dom (B) = \dom (A)$.   
Assume that the function $f$ is H\"older continuous on all finite intervals $[0,x_1]$, $x_1>0$, w.r.t. the $\Hte_T^\prime$-norm
(cf. \eqref{e asHol}), and that 
\begin{gather} \label{e asBsym}
B \quad \text{is similar to a self-adjoint operator in} \quad \Hte_T.
\end{gather}
Then problem \eqref{e TE}-\eqref{e TE BC0}-\eqref{e TE BCtau} has 
a unique $\Hte_T$-strong solution for each $\vphi \in \Hte_T$. 
\end{theorem}

\begin{proof}
Put $P_\pm := Q_\pm \uph \Hte_T$. Then
$P_+$ and $P_-$ are mutually complementary orthogonal projections 
in $\Hte_T$ and $J=P_+ - P_-$ is a signature operator.

Note that $L:=JB$ is a positive self-adjoint operator in $\Hte_T$
(that is, $B$ is J-positive and J-self-adjoint). Indeed,
since $L=(Q_+ - Q_-)T^{-1}A$ and $(Q_+ - Q_-) T^{\pm 1} = |T|^{\pm 1}$, we get
\begin{gather} \label{e L=A}
\btel  L  h, g  \bter_T = \btel |T| (Q_+ -Q_-) T^{-1} A h,  g  \bter = \btel   A h,  g  \bter , \qquad h \in \dom (A), \ g \in \Hte_T . 
\end{gather}
Hence (A1)--(A3) implies that $L=L^*>0$. 

So problem \eqref{e TE}-\eqref{e TE BCtau}
is reduced to correspondent problems for equation \eqref{e NonHomEq}.
Thus Theorem \ref{t TEFin} follows from Propositions \ref{p NHfin} and \ref{p NHinf}. 
\end{proof}

The form of the solutions is described by Propositions \ref{p NHfin}, \ref{p NHinf} 
and formula \eqref{e sol}.

\begin{remark} \label{r A3}
Actually, assumption (A3) is equivalent to the self-adjointness of the operator $L=(Q_+ - Q_-)B$ in $\Hte_T$. Note also that (A3) follows from (A1),(A2) and  \eqref{e asBsym}.
\end{remark}

Indeed, \eqref{e asBsym} implies that $L$ is closed and 
$\dom (L) (= \dom (B))$ is dense in $\Hte_T$. 
Assumption (A2) implies that $L>0$.
Assume that $L \neq L^*$.
Then $L$ admits a self-adjoint extension $\wt L \geq 0$, and $\wt L \supsetneqq L$.
Further, the operator $\wt B :=J \wt L$ is a J-non-negative J-self-adjoint extension of $B$.
Therefore \cite[Theorem II.3.25]{AzJ89} implies $\sigma_p (\wt B) \subset \R$. 
But \eqref{e asBsym} yields 
that $\ran (B - \la I) = \Hte_T$ for all $\la \in \C \setminus \R$. 
From this and $\wt B \supsetneqq B$, one gets $\C \setminus \R \subset \sigma_p (\wt B)$. 
This contradiction concludes the proof.

\section{Examples \label{s ex}}

If the spectrum $\sigma(B)$ is real and discrete, 
then assumption \eqref{a sim} is 
equivalent to the Riesz basis property for eigenfunctions of $B$.
For ordinary and partial differential operators with indefinite weights,
the Riesz basis property was studied in great detail 
(see \cite{KLKZ84,Pyat84,B85,CL89,F96,Pyat02,Par03} and references therein).
Below we consider several classes of differential equations with $B(=JL)$ such that $\sigma (B) \neq \overline{\sigma_{\disc} (B)}$. The theorems obtained in the previous sections are combined with known similarity results for Sturm-Liouville operators with an indefinite weight.
First, we consider in details a nonhomogeneous version of equation \eqref{e |mu|p}.  
Other applications will be indicated briefly (for homogeneous equations and the case $\tau < \infty$ only). Using Propositions \ref{p NHfin} and \ref{p NHinf}, one can extend these treatments on the half-space problems and the nonhomogeneous case.

\subsection{The equation $ (\sgn x)|\mu|^\alpha \psi_x= \psi_{\mu\mu}+f$}

Let us consider the equation 
\begin{gather} \label{e mupf}
(\sgn \mu) |\mu |^\alpha \frac {\partial \psi }{\partial x} (x,\mu ) = 
\frac {\partial^2 \psi}{\partial \mu^2} (x, \mu ) 
+ f (x,\mu)  
\qquad ( 0<x<\tau \leq \infty, \ \mu \in \R ), 
\end{gather}
where 
$\alpha > -1$ is a constant. 
In the case $\tau<\infty$, the associated boundary conditions 
take the form
\begin{gather} \label{t mupBCfin}
\psi (0, \mu) = \vphi (\mu) \qquad \text{if} 
\quad \mu >0 , \qquad
\psi (\tau, \mu) = \vphi (\mu) \qquad \text{if} 
\quad \mu <0. 
\end{gather}
If $\tau=\infty$, we should change them to
\begin{gather} \label{t mupBCinf}
\quad \psi (0, \mu) = \vphi (\mu) \qquad \text{if} 
\quad \mu >0 , \qquad
\int_\R | \psi (x, \mu) |^2 |\mu|^\alpha d\mu  = O(1) \quad \text{as} \quad x \to +\infty.
\end{gather}

To write \eqref{e mupf} in the form \eqref{e TE}, one can put $\Hte = L^2 (\R)$,
$(T y) (\mu) = (\sgn \mu) |\mu |^\alpha y (\mu)$ 
and \\
$A : y \mapsto  -y'' $. 
Then, 
\[
\Hte_T = L^2 (\R, |\mu|^\alpha d\mu), \qquad \Hte_T^\prime = L^2 (\R, |\mu|^{-\alpha} d\mu).  
\]
It is assumed that $A$ is an operator from $\Hte_T $ to
$\Hte_T^\prime $ and that it is defined on the natural domain
\[
\dom (A) = \{y \in \Hte_T: y,y' \in AC_{\loc} (\R) \quad \text{and} \quad y'' \in \Hte_T^\prime \}.
\]
One can find the operators $Q_\pm$ (see Section \ref{s AKE}) and check that $J:= (Q_+ - Q_-) \uph \Hte_T$ coincides with $J$ defined by \eqref{e HJ}. Consider the operators $B:=T^{-1} A$ and $L:=JB$.
Both operators are defined on $\dom (A)$ by the following
differential expressions
\[
B : y(\mu) \mapsto  -\frac {(\sgn \mu)}{ |\mu |^\alpha} y''(\mu) 
,
\qquad L : y(\mu) \mapsto  -\frac 1{ |\mu |^\alpha} y''(\mu) 
.
\]
Clearly, $By, Ly \in \Hte_T$ for all $y \in \dom (A)$. So $B$ and $L$ are operators in $\Hte_T$. 
It is easy to check (see e.g. \cite{FN98}) that 
\begin{gather} \label{e Lp>0}
L \quad \text{is a positive self-adjoint
operator in} \quad  \Hte_T.
\end{gather}
It follows from \eqref{e Lp>0}, Remark \ref{r A3} and  \eqref{e L=A} that conditions (A1)-(A3) from Section \ref{s AKE} are fulfilled for the operator $A$.
It was proved in \cite{FN98} (see also \cite{CN96,KarMFAT00,KosMFAT06})
that the operator $B$ is similar to a self-adjoint
operator in the Hilbert space $L^2 (\R, |\mu|^\alpha d\mu)$.
By Theorem \ref{t TEFin}, we obtain the following result. 

\emph{Assume that $\int_\R |f(x,\mu)|^2 |\mu|^{-\alpha} d\mu < \infty$, $x \in (0,\tau)$, and that $f(x)$ is H\"older continuous on all finite intervals $[0, x_1 ]$ (see \eqref{e asHol}) 
as a function with values in $L^2 (\R, |\mu|^{-\alpha} d\mu)$; 
then
there is a unique solution of problem \eqref{e mupf}, \eqref{t mupBCfin} 
(problem \eqref{e mupf}, \eqref{t mupBCinf}, in the case $\tau=+\infty$) for every $\vphi \in L^2 (\R, |\mu|^\alpha d\mu)$.
}

\begin{remark}
In the case $\tau < \infty$ and $\alpha=0$, problem \eqref{e mupf}-\eqref{t mupBCfin} was considered in \cite{T85} under additional assumptions that 
$\vphi$ belongs to a certain H\"older class. The half-space problem ($\tau = \infty$)  was studied in \cite{Pag74,G75} (see also remarks in \cite[Appendix II]{Pag76}). More precisely, in \cite{Pag74}, the homogeneous equation was considered for all $\alpha>-1$ under the assumption 
$\int_{\R_+} ( |\vphi 
|^2 |\mu|^{\alpha}+|\vphi'
|^2) d \mu < \infty. 
$
In \cite{G75}, the nonhomogeneous case was considered for $\alpha=1$ and 
$\vphi (\cdot )$, $f(\cdot )$ from certain classes of continuous functions. Explicit integral representations for solutions were obtained in \cite{Pag74,G75,T85}. Note also that Eq. \eqref{e mupf}
was studied in \cite{BG68} for $\mu$ in a finite interval $[-a,a]$, however, the latter makes the spectrum of $B$ discrete. 
\end{remark}

\subsection{The case when $L$ is uniformly positive}

Consider equation \eqref{e PsiEq} under the assumption $L=L^*>\delta>0$, i.e.,
the operator $L$ is uniformly positive in the Hilbert space $H$. 
As before, put $B=JL$, where $J$ is a signature operator in $H$. 
In this case, 
$B$ is similar to a self-adjoint operator iff $\infty$
is not a singular critical point of $B$ (see Proposition \ref{p cr=sim}). 
For ordinary differential operators with indefinite weights, the regularity of the critical point $\infty$ is well studied even in the case of a finite number of turning points (i.e., the points where the weight $w$ changes sign).
We will use one result that follows from \cite{CL89}.

Let the functions $w,p,q$ be such that
\begin{gather} \label{e coef}
w, q \in L^1_{\loc} (\R), \qquad p(\mu)>0 \quad \text{a.e. on} \quad \R, \qquad \text{and} \quad p^{-1}, p \in L^\infty_{\loc} (\R).
\end{gather}
Assume that the maximal operator
\begin{gather*} 
\quad L:y \mapsto \frac 1{|w|} (-(py')'+qy) \quad \text{is self-adjoint in the Hilbert space} \quad  L^2 (\R, |w(\mu)| d\mu),
\end{gather*}
i.e., it is in the limit point case both at $+\infty$ and $-\infty$.
Assume also that the sets
\[
\mathcal{I}_+ := \{ \mu \in \R : w(\mu) >0\} \quad \text{and} \quad
\mathcal{I}_- := \{ \mu \in \R : w(\mu) <0\}
\]
are both of positive Lebesgue measure. 

The elements of the set
$\overline{\mathcal{I}_+} \cap \overline{\mathcal{I}_-} $
are called
\emph{turning points of} $w$.

Put $(Jf)(\mu):= (\sgn w(\mu)) f(\mu)$ for $f \in  L^2 (\R, |w(\mu)| d\mu)$,
and $B=JL$. Then 
\[
B:y \mapsto \frac 1{w} (-(py')'+qy) 
\]
is a J-self-adjoint operator in $L^2 (\R, |w(\mu)| d\mu)$.

The following definition is an improved version of Beals' condition \cite{B85}.

\begin{definition}[\cite{CL89}] \label{d simple}
A  function $w$ is said to be
\emph{simple from the right at} $\mu_0$ if there exist $\delta>0$ such that $w$ is nonnegative (nonpositive) on $[\mu_0,\mu_0+\delta]$ and 
\begin{gather} \label{e simple}
w (\mu)= (\mu -\mu_0)^{\beta} \rho (\mu) \qquad
(w (\mu)= -(\mu -\mu_0)^{\beta} \rho (\mu), \ \text{respectively})
\end{gather}
holds a.e. on $[\mu_0,\mu_0+\delta]$ with some $\beta >-1$,
$\rho \in C^1 [\mu_0, \mu_0+\delta]$, $\rho(\mu_0)>0$.
A function $w$ is said to be \emph{simple from the left at} $\mu_0$ if
the function $\mu \mapsto w(-(\mu - \mu_0)+\mu_0)$ is simple from the right at
$\mu_0$. A function $w$ is said to be \emph{simple at} $\mu_0$ if it is simple from the right and simple from the left at $\mu$ (with, possibly, different numbers $\beta$).
\end{definition}

It follows from the results of \cite{CL89} that $\infty$ is a regular critical point of $B$ if the following assumptions are fulfilled:
\begin{description}
\item[(i)] the function $w$ has a finite number of turning points at which it is simple,
\item[(ii)] $L=L^*\geq0$ and $\rho(B)\neq \emptyset$. 
\end{description}

For the proof, we note that the set $\mathcal{D}[JB]=\mathcal{D}[L]$ is separated 
(in terms of \cite[Subsection 3.2]{CL89}) since $L$ is in the limit point case at $+\infty$ and $-\infty$. So the case (i) of \cite[Theorem 3.6]{CL89} holds for $B$.

If, additionally, $L>\delta>0$ then $0 \in \rho(L)$ and, consequently, $0 \in \rho(B)$ since $J$ is a unitary operator.
Besides, $0$ is not a critical point of $B$. Hence one can obtain the following statement from Proposition \ref{p cr=sim} and \cite[Theorem 3.6]{CL89}.

\begin{proposition} \label{p CPSL}
Assume that the function $w$ has a finite number of turning points at which it is simple, and that $L=L^*>\delta> 0$. Then the operator $B$ is similar to a self-adjoint operator in
$L^2 (\R, |w(\mu)| d\mu)$.
\end{proposition}

Thus, if the conditions of Proposition \ref{p CPSL} are satisfied,  
Theorem \ref{t EU} implies that there is a unique strong solution of the corresponding 
boundary problem \eqref{e PsiEq}-\eqref{e PsiBC} (cf. \cite[Section 4]{GGvdM88}).

Note that the condition $L>\delta>0$ is fulfilled whenever 
\begin{equation} \label{e q/w}
\frac {q(\mu)} {|w(\mu)|} > C >0, \quad \text{where} \ C \ \text{is a constant (cf. \cite[Appendix I]{Pag76})}.
\end{equation}

\begin{example}[cf. \cite{PagT71}]
Consider the equation 
\begin{gather} \label{e mupfk}
(\sgn \mu) |\mu |^\alpha \frac {\partial \psi }{\partial x} = 
\frac {\partial^2 \psi}{\partial \mu^2} - k \psi 
\qquad ( 0<x<\tau < \infty, \ \mu \in \R ), 
\end{gather}
where $k > 0$ and $\alpha \in (-1,0]$ are constants. 
Evidently, the weight function $w(\mu)=(\sgn \mu) |\mu |^\alpha$ is simple in its only turning point $0$. Besides, condition \eqref{e q/w} is satisfied since $\alpha \in (-1,0]$.
So problem \eqref{e mupfk}, \eqref{t mupBCfin} has a unique solution 
for every $\vphi \in L^2 (\R, |\mu|^\alpha d\mu)$.
\end{example}

\begin{remark}
The half-space problem for the equation \eqref{e mupfk} with $\alpha=1$ and $k>0$ was studied in 
\cite{Pag75}. Note that if $\alpha>0$, the operator $L$ associated with Eq. \eqref{e mupfk} is not uniformly positive. To extend the method of the present paper to the case $\alpha>0$,
one should prove that $0$ is a regular critical point of the operator $B:y \mapsto  (\sgn \mu) |\mu |^{-\alpha} (-y''+ky) $ (cf.  Subsection \ref{ss FPe}).
\end{remark}

Improvements of condition \eqref{e simple} may be found in 
\cite{CL89} (the end of Subsection 3.1) and \cite{F96}. Note also that \cite[Theorem 3.6]{CL89} 
is valid for higher order ordinary differential operators.


\subsection{The Fokker-Plank equation}\label{ss FPe}


In the case when $inf \, \sigma_{ess} (L) =0$, the similarity problem for the operator 
$B:y \mapsto \frac 1{w} (-(py')'+qy)$ 
is more difficult.
This question was considered in 
\cite{CN96,FN98,KarMFAT00,FSh02,KarM04,KosMFAT06,KarM06,KosMZ06,KarKos06} (see also references therein). 
A general method was developed in 
\cite{KarM04,KarM06},
where the operator $B$ with \\
$w(\mu)= \sgn \mu$ and $p(\mu) \equiv 1$ was studied.
However, all results contained in \cite[Sections 3-6]{KarM06} are valid without changes in proofs for the general Sturm-Liouville operator $B$ with one turning point.
The approach of \cite{KarM06} was applied to the case $q\equiv 0$ in \cite{KosMZ06,KosDis,KarKosM06},
where the following theorem was proved.

\begin{theorem}[\cite{KosMZ06,KosDis,KarKosM06}] \label{t Kos}
Assume that $w \in L^1_{\loc} (\R)$, $\mu w (\mu)>0$  for a.a. $\mu \in \R$,
and the function $w$ is simple at its only turning point $0$.
Assume also that  
$w$ has the form $w(\mu) = \pm r(\mu) |\mu|^{\alpha_\pm}$ 
for $\mu \in \R_\pm$, 
where the function $r$ satisfies the following conditions 
\begin{gather} \label{e rCond}
\int_1^{+\infty} \mu^{\alpha_+/2} |r(\mu) - c_+| d\mu < \infty, \quad
\int_{-\infty}^{-1} |\mu|^{\alpha_-/2} |r(\mu) - c_-| d\mu < \infty,
\end{gather}
and $\alpha_\pm >-1$, $c_\pm >0$ are constants.
Then the operator $B:y \mapsto -\frac 1{w} y''$ defined on its maximal domain in 
$L^2 (\R, |w(\mu)| d\mu)$ is similar to a self-adjoint operator.
\end{theorem}

Evidently, the operator $L:y \mapsto -\frac 1{|w|} y''$ is nonnegative in $L^2 (\R, |w(\mu)| d\mu)$. Moreover, if condition 
\eqref{e rCond} is fulfilled,
$L$ is positive. Indeed, any solution of the equation $Ly_0=0$ has the form $y_0(\mu) = c_1+c_2 \mu$, $\mu \in \R$, where $c_1$ and $c_2$ are constants. But it follows from 
\eqref{e rCond} that  $y_0 \not \in L^2 (\R, |w(\mu)| d\mu)$ for all $c_1,c_2 \in \C$.

Applying Theorem \ref{t EU}, we obtain the existence and uniqueness theorem 
for the time-independent Fokker-Plank equation of the simplest kind (see e.g. \cite{S95} and also references in \cite{BP83})
\begin{gather}
\mu \frac {\partial \psi }{\partial x} (x,\mu ) =  
b(\mu) \frac {\partial^2 \psi}{\partial \mu^2} (x, \mu ) 
\qquad ( 0<x<\tau<\infty, 
\ \mu \in \R ).
\label{e FP} 
\end{gather}
Namely, \emph{if the assumptions of Theorem \ref{t Kos} are satisfied for the function $w(\mu)=b(\mu)\mu^{-1}$, then the boundary value problem \eqref{e FP}, \eqref{t mupBCfin} 
has a unique solution for every $\vphi \in L^2 (\R, |w(\mu)| d\mu)$.} 

\begin{remark}
The case when $B$ is an ordinary differential operator of higher order and 
$inf \, \sigma_{ess} (L) =0$ was considered in \cite{CN96,KarMFAT00,KarMN00}.
For partial differential operators, see \cite{CN94,FL96,CN98,HKar05}.
\end{remark}

\textit{The author express his gratitude to V.A. Derkach, who drew author's 
attention to the papers \cite{B85,KvdMP87}; to M.M. Malamud, V.A.Marchenko, and H. Stephan for stimulating discussions about this circle of problems.
The author wishes to thank for the hospitality to the University of Zurich, 
the Johann Radon Institute for Computational and Applied Mathematics, and the University of Calgary, where various parts of this paper have been written.}

\quad\\
I. Karabash, 
\\Department of Mathematics and Statistics, University of Calgary, 2500 University Drive NW Calgary T2N 1N4, Alberta, Canada \\ 
and \\
Department of PDE, Institute of Applied Mathematics and Mechanics, R. Luxemburg str. 74, Donetsk 83114, Ukraine\\
e-mail: karabashi@yahoo.com, karabashi@mail.ru, karabash@math.ucalgary.ca
\end{document}